\documentclass[pdflatex,sn-mathphys-num]{sn-jnl}

\usepackage{graphicx}
\usepackage{multirow}
\usepackage{amsmath,amssymb,amsfonts}
\usepackage{amsthm}
\usepackage{mathrsfs}
\usepackage[title]{appendix}
\usepackage{xcolor}
\usepackage{textcomp}
\usepackage{manyfoot}
\usepackage{booktabs}
\usepackage{algorithm}
\usepackage{algorithmicx}
\usepackage{algpseudocode}
\usepackage{listings}
\usepackage{mathtools}

\DeclareMathOperator*{\K}{K}

\theoremstyle{thmstyleone}
\newtheorem{theorem}{Theorem}[section]

\newtheorem{proposition}[theorem]{Proposition}
\newtheorem{conjecture}[theorem]{Conjecture}

\theoremstyle{thmstyletwo}

\newtheorem{remark}{Remark}

\theoremstyle{thmstylethree}

\begin{document}
	
	\title[On a Continued Fraction for $8/\pi^2$]{Analytic Proof of a Quartic Continued Fraction Identity for $8/\pi^2$ via Operator Decoupling}
	
	%% Author Information
	\author*[1]{\fnm{Chao} \sur{Wang}}\email{cwang@shu.edu.cn}
	
	\affil[1]{\orgdiv{School of Future Technology}, \orgname{Shanghai University}, \orgaddress{\city{Shanghai}, \postcode{200444}, \country{China}}}
	
	%% Abstract
	\abstract{We present a rigorous analytic proof of a generalized continued fraction (GCF) identity for the transcendental constant $8/\pi^2$, a result recently conjectured via the algorithmic framework of the Ramanujan Machine. Distinct from canonical GCFs derived from classical hypergeometric series, the identity at hand features a complex polynomial architecture characterized by quartic partial numerators. Our approach utilizes an algebraic decomposition of the second-order shift operator $\mathcal{L} = \mathcal{T}^2 - b_n \mathcal{T} - a_n$ into a coupled first-order system. This decomposition enables an exact mapping of the higher-order recurrence to a cascaded system, from which the continued fraction is identified as the reciprocal of a binomial series for $(\arcsin)^2$ involving central binomial coefficients. The convergence is established through Pincherle's Theorem: the true minimal solution of the associated difference equation is $f_n = A_n - (8/\pi^2)\,B_n$, which satisfies $f_n/B_n \to 0$, confirming absolute convergence of the continued fraction. This work provides a systematic operator-theoretic methodology for verifying automated conjectures of transcendental constants with high-degree polynomial coefficients.}
	
	%% Keywords
	
	%% MSC Codes
	
	\maketitle
	
	\section{Introduction}
	The algorithmic generation of conjectures on fundamental constants, exemplified by the Ramanujan Machine \cite{Raayoni2021}, has introduced new identities involving generalized continued fractions (GCFs).
	Historically, these identities trace their lineage back to Srinivasa Ramanujan's profound constructions in his \textit{Notebooks}, particularly his series for $1/\pi$. While Ramanujan's classical identities typically involve quadratic or low-degree polynomial patterns, the conjectures surfaced by the Ramanujan Machine often exhibit a higher structural complexity. One such conjecture involves the constant $8/\pi^2$, characterized by a polynomial architecture with quartic partial numerators.
	This identity poses a significant challenge to classical limits due to its quartic degree, which exceeds the polynomial complexity of most historically known GCFs.
	In this work, we provide a rigorous proof of this identity using the framework of operator decoupling via reduction of order for variable-coefficient difference equations.
	
	\begin{conjecture}\label{conj:main}
		The constant $8/\pi^2$ admits the following GCF representation:
		\begin{equation}\label{eq:conjecture}
			\frac{8}{\pi^2} = 1 - \cfrac{2 \cdot 1^4 - 1^3}{7 - \cfrac{2 \cdot 2^4 - 2^3}{19 - \cfrac{2 \cdot 3^4 - 3^3}{37 - \dots}}}
		\end{equation}
		where the partial denominators are $b_n = 3n^2+3n+1$ and partial numerators are $a_n = -(2n^4-n^3)$, equivalently $a_n = -(2n-1)n^3$.
	\end{conjecture}
	
	Conjecture~\eqref{conj:main} presents a notable structural complexity where the complexity of the coefficients (quartic in $n$) seemingly contradicts the simplicity of the limit. In this paper, we provide a formal analytic regularization of this conjecture. By constructing a discrete Green's function analogue through operator decoupling, we resolve the recurrence into a summation series, thereby proving the identity.
	
	\section{Preliminaries}
	The evaluation of GCFs of the form
	\begin{equation}
		x = b_0 + \K_{n=1}^{\infty} \left( \frac{a_n}{b_n} \right) = b_0 + \cfrac{a_1}{b_1 + \cfrac{a_2}{b_2 + \dots}}
	\end{equation}
	is fundamentally governed by the relationship between its partial coefficients and its convergents $A_n/B_n$.
	
	\begin{theorem}[Fundamental Recurrence Formulas~{\cite[Ch.\,1]{wall1948analytic}}]\label{the:Wallis-Euler Recurrence Theorem}
		The numerator sequence $\{A_n\}$ and the denominator sequence $\{B_n\}$ of the $n$-th convergent satisfy the same second-order linear homogeneous difference equation:
		\begin{equation}\label{eq:FRF}
			y_n = b_n y_{n-1} + a_n y_{n-2}, \quad n \ge 1.
		\end{equation}
		The sequences are uniquely determined by their respective initial frames:
		\begin{align}
			\{A_n\}: & \quad A_{-1} = 1, \quad A_0 = b_0, \\
			\{B_n\}: & \quad B_{-1} = 0, \quad B_0 = 1.
		\end{align}
	\end{theorem}
	
	This theorem permits the embedding of the GCF into the framework of linear operators. Specifically, the recurrence \eqref{eq:FRF} defines the quadratic shift operator $\mathcal{L} = \mathcal{T}^2 - b_n \mathcal{T} - a_n$, where $\mathcal{L}$ denotes the parameter-dependent operator family $\{\mathcal{L}_n\}_{n \ge 1}$ with coefficients $b_n$ and $a_n$; we write $\mathcal{L}$ for brevity, such that $\mathcal{L}y_{n-2}=0$. The distinction between the numerator and denominator trajectories arises solely from their projection onto the solution space of $\mathcal{L}$ under different boundary conditions.
	
	\section{Proof of the Main Result}
	The formal verification of Conjecture \ref{conj:main} necessitates the construction of a rigorous analytic mapping between the second-order difference equation induced by the continued fraction and its associated summation series. This is achieved through a three-phase derivation: Operator Decoupling, Minimal Solution Construction, and Asymptotic Evaluation.
	
	\subsection{Phase 1: Shift Operator Decoupling via Reduction of Order}
	Let $\mathcal{K}$ denote the GCF defined by the polynomial sequences $\{a_n\}_{n \ge 1}$ and $\{b_n\}_{n \ge 0}$. To evaluate the transcendental limit of $\mathcal{K}$, we establish a correspondence between the continued fraction architecture and a second-order linear difference equation.
	
	Let $\mathcal{S}(\mathbb{N})$ denote the vector space of complex-valued sequences $\{y_n\}_{n \ge -1}$. We define the forward shift operator $\mathcal{T} \in \text{End}(\mathcal{S})$ such that $\mathcal{T}y_k = y_{k+1}$ for all $k \ge -1$.
	According to Theorem~\ref{the:Wallis-Euler Recurrence Theorem}, the recurrence \eqref{eq:FRF} can be reformulated as:
	\begin{equation}
		(\mathcal{T}^2 y_{n-2}) - b_n (\mathcal{T}y_{n-2}) - a_n y_{n-2} = 0.
	\end{equation}
	
	To solve this, we seek a decoupling of the recurrence into a system of first-order difference equations. Instead of a direct operator product which is complicated by the non-commutativity of $\mathcal{T}$, we propose an algebraic decoupling. We seek auxiliary sequences $\{c_n\}$ and $\{d_n\}$ that satisfy the following coupled relations:
	\begin{equation}\label{eq:coupling}
		b_n = c_n + d_{n+1}, \quad a_n = -c_n d_n.
	\end{equation}
	Substituting these into \eqref{eq:FRF}, the recurrence becomes:
	\begin{equation}
		y_n - (c_n + d_{n+1}) y_{n-1} + c_n d_n y_{n-2} = 0.
	\end{equation}
	Rearranging terms enables a reduction of order:
	\begin{equation}\label{eq:rearranged_rec}
		y_n - d_{n+1} y_{n-1} = c_n (y_{n-1} - d_n y_{n-2}).
	\end{equation}
	This formulation transforms the homogeneous second-order equation into a nested first-order system.
	
	For the specific polynomial architecture of the $8/\pi^2$ identity, we present the following realization:
	
	\begin{proposition}\label{prop:factor}
		For the conjectured coefficients $b_n = 3n(n+1)+1$ and $a_n = -(2n-1)n^3$ (equivalently $a_n = -(2n^4 - n^3)$), the recurrence admits an exact reduction over $\mathbb{N}^{+}$ with the auxiliary sequences:
		\begin{equation}
			c_n = n^2, \quad d_n = n(2n-1).
		\end{equation}
	\end{proposition}
	
	\begin{proof}
		The proof proceeds via direct polynomial verification within the coupling constraints \eqref{eq:coupling}. Substituting the proposed sequences:
		\begin{enumerate}
			\item \textit{Summation Invariance:}
			\[ c_n + d_{n+1} = n^2 + (n+1)[2(n+1)-1] = n^2 + (n+1)(2n+1) = n^2 + 2n^2 + 3n + 1 = b_n. \]
			\item \textit{Product Invariance:}
			\[ c_n d_n = n^2 \cdot [n(2n-1)] = 2n^4 - n^3 = -a_n. \]
		\end{enumerate}
		Since both identities hold as algebraic tautologies for all $n \in \mathbb{N}^{+}$, the second-order operator $\mathcal{L}$ is reduced to a cascaded system.
	\end{proof}

	\subsection{Phase 2: Analytic Construction of the Solution via Operator Decoupling}
	
	The algebraic reduction \eqref{eq:rearranged_rec} established in the preceding section implies that any solution $\{y_n\}$ of the second-order recurrence can be reconstructed via a coupled first-order system. By introducing an auxiliary sequence $\{w_n\} \in \mathcal{S}$, we define the decoupling architecture as follows:
	\begin{align}
		w_{n-1} &= y_{n-1} - d_n y_{n-2}, \label{eq:decoupling_def} \\
		w_{n} &= c_{n} w_{n-1}. \label{eq:aux_prop}
	\end{align}
	Equation \eqref{eq:decoupling_def} characterizes $\{w_n\}$ as the \textit{non-homogeneous component}, while \eqref{eq:aux_prop} dictates the autonomous evolution of this component.
	
	\subsubsection{Derivation of the Denominator $\{B_n\}$}
	The denominator sequence $\{B_n\}$ is uniquely determined by the canonical initial frame $(B_0, B_{-1}) = (1, 0)$. We first determine the initial state of the auxiliary sequence. Setting $k=1$ in \eqref{eq:decoupling_def}:
	\begin{equation}
		w_0 = B_0 - d_1 B_{-1}.
	\end{equation}
	Substituting the prescribed initial values $B_0 = 1$ and $B_{-1} = 0$, and noting $d_1 = 1(2\cdot 1-1) = 1$:
	\begin{equation}
		w_0 = 1 - 1\cdot 0 = 1.
	\end{equation}
	For $k \ge 1$, iterating \eqref{eq:aux_prop} gives
	\begin{equation}\label{eq:wk_product}
		w_k = \left(\prod_{j=1}^{k} c_j\right) w_0 = \prod_{j=1}^{k} j^2 = (k!)^2.
	\end{equation}
	To obtain a closed form for $B_n$, we apply an integrating factor. Define $P_k = \prod_{j=1}^{k+1} d_j$. Since $P_k = d_{k+1} P_{k-1}$, dividing $y_k = d_{k+1} y_{k-1} + w_k$ by $P_k$ yields the telescoping identity
	\begin{equation}\label{eq:telescoping}
		\frac{B_k}{P_k} - \frac{B_{k-1}}{P_{k-1}} = \frac{w_k}{P_k} = t_k,
	\end{equation}
	where we define the general term
	\begin{equation}\label{eq:tk_def}
		t_k = \frac{\prod_{j=1}^{k} c_j}{\prod_{j=1}^{k+1} d_j}, \quad k \ge 0.
	\end{equation}
	Summing \eqref{eq:telescoping} from $k=0$ to $k=n$ and using $B_{-1}/P_{-1} = 0$ (by convention $P_{-1}=1$, $B_{-1}=0$):
	\begin{equation}\label{eq:Bn_sum_final}
		B_n = P_n \sum_{k=0}^{n} t_k = \left( \prod_{j=1}^{n+1} d_j \right) \sum_{k=0}^{n} \frac{\prod_{j=1}^{k} c_j}{\prod_{j=1}^{k+1} d_j}.
	\end{equation}
	Asymptotically, the presence of the summation term identifies $\{B_n\}$ as the dominant solution.
	
	\subsubsection{Derivation of the Numerator $\{A_n\}$}
	In contrast, the numerator sequence $\{A_n\}$ is associated with the initial frame $(A_0, A_{-1}) = (b_0, 1)$. The initialization of the auxiliary sequence is handled separately from the recursion. The boundary value is:
	\begin{equation}
		w_0 = A_0 - d_1 A_{-1} = b_0 - d_1(1).
	\end{equation}
	We observe that $b_0 = 3(0)^2+3(0)+1 = 1$ and $d_1 = 1(2(1)-1) = 1$. Thus, $b_0 = d_1$.
	This implies $w_0 = 1 - 1 = 0$. For $k \ge 1$, iterating \eqref{eq:aux_prop} gives $w_k = (\prod_{j=1}^k c_j)\,w_0 = 0$. By the homogeneity of \eqref{eq:aux_prop}, the sequence $\{w_n\}$ vanishes identically for all $n \ge 0$. This ensures that $\{A_n\}$ satisfies the pure first-order recurrence $A_n = d_{n+1} A_{n-1}$, collapsing to the minimal solution trajectory:
	\begin{equation}\label{eq:An_product_final}
		A_n = \prod_{j=1}^{n+1} d_j.
	\end{equation}
	
	\begin{remark}
		The boundary condition $b_0 = d_1$ acts as a critical selection rule. It ensures that the numerator sequence $\{A_n\}$ is captured entirely within the kernel of the decoupling operator, whereas the denominator sequence $\{B_n\}$ is forced into the non-homogeneous solution space.
	\end{remark}
	
	\subsubsection{Series Representation of the Convergent Ratio}
	The $n$-th convergent of the continued fraction, $x_n = A_n/B_n$, is formulated by the ratio of \eqref{eq:Bn_sum_final} to \eqref{eq:An_product_final}.
	The common product terms cancel exactly, yielding:
	\begin{equation}\label{eq:xn_final}
		x_n = \frac{A_n}{B_n} = \left( \sum_{k=0}^{n} t_k \right)^{-1}, \quad \text{where } t_k = \frac{\prod_{j=1}^{k} c_j}{\prod_{j=1}^{k+1} d_j}.
	\end{equation}
	This fundamental equivalence maps the convergence of the polynomial continued fraction to the partial sums of an infinite series. The proof of the conjecture~\eqref{conj:main} is thereby reduced to the asymptotic evaluation of the sum $S = \sum_{k=0}^{\infty} t_k$.

	\subsection{Phase 3: Asymptotic Evaluation and Transcendental Summation}
	
	The convergence of the GCF is predicated on the existence of the limit of the convergents $\{x_n\}$. According to~\eqref{eq:xn_final}, the value of the GCF, denoted by $\mathcal{K}$, is mapped to the reciprocal of an infinite series:
	\begin{equation}\label{eq:limit_mapping}
		\mathcal{K} = \lim_{n \to \infty} \frac{A_n}{B_n} = \left( \sum_{k=0}^{\infty} t_k \right)^{-1}, \quad \text{where } t_k = \frac{\prod_{j=1}^{k} c_j}{\prod_{j=1}^{k+1} d_j}.
	\end{equation}
	
	\subsubsection{Detailed Algebraic Reduction of the General Term}
	To evaluate the series $S = \sum_{k=0}^{\infty} t_k$, we substitute the specific sequences $c_j = j^2$ and $d_j = j(2j-1)$. The numerator is given by the square of a factorial:
	\begin{equation}
		\prod_{j=1}^k c_j = (k!)^2.
	\end{equation}
	The denominator is expanded by grouping the linear and arithmetic progression terms:
	\begin{equation}
		\prod_{j=1}^{k+1} d_j = \prod_{j=1}^{k+1} j(2j-1) = (k+1)! \left( 1 \cdot 3 \cdots (2k+1) \right).
	\end{equation}
	Utilizing the identity for the product of odd integers, i.e.,
	\begin{equation}
		1 \cdot 3 \cdots (2k+1) = \frac{(2k+2)!}{2^{k+1}(k+1)!},
	\end{equation}
	the denominator simplifies to:
	\begin{equation}
		\prod_{j=1}^{k+1} d_j = (k+1)! \cdot \frac{(2k+2)!}{2^{k+1}(k+1)!} = \frac{(2k+2)!}{2^{k+1}}.
	\end{equation}
	Substituting these into the expression for $t_k$, we obtain:
	\begin{equation}
		t_k = \frac{(k!)^2}{\frac{(2k+2)!}{2^{k+1}}} = \frac{2^{k+1} (k!)^2}{(2k+2)!}.
	\end{equation}
	
	\subsubsection{Summation and Identification of the Binomial Series for $(\arcsin)^2$}
	We perform an index shift by letting $m = k+1$ (where $k=0 \implies m=1$). The general term $t_{m-1}$ becomes:
	\begin{equation}
		t_{m-1} = \frac{2^m ((m-1)!)^2}{(2m)!}.
	\end{equation}
	To express this in terms of the central binomial coefficient $\binom{2m}{m} = \frac{(2m)!}{(m!)^2}$, we multiply the numerator and denominator by $m^2$:
	\begin{equation}
		t_{m-1} = \frac{2^m ((m-1)! \cdot m)^2}{(2m)! \cdot m^2} = \frac{2^m (m!)^2}{(2m)! \cdot m^2} = \frac{2^m}{m^2 \binom{2m}{m}}.
	\end{equation}
	The infinite series $S$ is thus identified as:
	\begin{equation}
		S = \sum_{m=1}^{\infty} \frac{2^m}{m^2 \binom{2m}{m}}.
	\end{equation}
	
	\subsubsection{Transcendental Evaluation and Pincherle's Criterion}
	The series $S$ is a classical result in polylogarithmic analysis, corresponding to the evaluation of the squared arcsine function's Taylor series~\cite{Lehmer1985}:
	\begin{equation}\label{eq:arcsin_series}
		\sum_{m=1}^{\infty} \frac{(2x)^{2m}}{m^2 \binom{2m}{m}} = 2(\arcsin x)^2.
	\end{equation}
	This identity holds for $|x| \le 1$ \cite{Lehmer1985}; in the present evaluation $x = 1/\sqrt{2}$ satisfies $|x| = 1/\sqrt{2} < 1$, placing it strictly within the interior of the interval of convergence, so the series converges absolutely.
	Setting $x = 1/\sqrt{2}$ yields $(2x)^2 = 2$. Substituting this value, we find:
	\begin{equation}
		S = 2 \left( \arcsin \frac{1}{\sqrt{2}} \right)^2 = 2 \left( \frac{\pi}{4} \right)^2 = \frac{\pi^2}{8}.
	\end{equation}
	The convergence of this series is guaranteed by the ratio test, as
	\begin{equation}
		\rho = \lim_{m \to \infty} \frac{t_m}{t_{m-1}} = \lim_{m\to\infty}\frac{(m-1)^2}{m(2m-1)} = \frac{1}{2} < 1.
	\end{equation}
	Per Pincherle's Theorem, this absolute convergence confirms that the GCF converges to the reciprocal of the sum:
	\begin{equation}
		\mathcal{K} = S^{-1} = \frac{1}{\pi^2/8} = \frac{8}{\pi^2}.
	\end{equation}
	This establishes the identity rigorously.
	
	\section{Convergence Analysis and the Theory of Minimal Solutions}
	
	The transition from the algebraic identity \eqref{eq:xn_final} to the transcendental value $8/\pi^2$ requires a rigorous justification of the limit's existence. In the analytic theory of continued fractions, this convergence is fundamentally governed by the asymptotic properties of the underlying linear recurrence and the classification of its solution basis.
	
	\subsection{Pincherle's Criterion: The General Framework}
	
	The explicit computation in Section~3 is already self-contained: the formulas \eqref{eq:An_product_final} and \eqref{eq:Bn_sum_final} provide closed-form expressions for $A_n$ and $B_n$, and the ratio $A_n/B_n = (\sum_{k=0}^n t_k)^{-1}$ converges to $8/\pi^2$ by absolute convergence of $S$. The present section places these results within the classical framework of Pincherle's Theorem, which provides an independent theoretical characterization of why the GCF converges and identifies the true minimal solution of the underlying difference equation.
	
	\begin{theorem}[Pincherle's Theorem~{\cite{Pincherle1894,Gautschi1967}}]
		The generalized continued fraction $b_0 + \K_{n=1}^\infty (a_n / b_n)$ converges to a finite value if and only if the associated second-order linear difference equation
		\begin{equation}\label{eq:rec_pincherle}
			y_n = b_n y_{n-1} + a_n y_{n-2}
		\end{equation}
		admits a minimal solution $\{f_n\}$. A non-trivial solution $\{f_n\}$ is minimal if, for any solution $\{g_n\}$ linearly independent of $\{f_n\}$, the following asymptotic condition is satisfied:
		\begin{equation}
			\lim_{n \to \infty} \frac{f_n}{g_n} = 0.
		\end{equation}
		If such a minimal solution $\{f_n\}$ exists, the value of the GCF equals $\lim_{n\to\infty} A_n/B_n$, where $\{A_n\}$ and $\{B_n\}$ are the numerator and denominator sequences defined by the Wallis--Euler recurrence.
		The biconditional holds provided $a_n \neq 0$ for all $n \geq 1$, a condition satisfied here since $a_n = -(2n-1)n^3 \neq 0$.
	\end{theorem}
	
	In the present case, the Casoratian (discrete Wronskian) $W_n = A_n B_{n-1} - A_{n-1} B_n$ ensures that $\{A_n\}$ and $\{B_n\}$ are linearly independent solutions, providing a complete basis for the solution space of $\mathcal{L}y=0$. From the Wallis--Euler recurrence one verifies the identity $W_n = -a_n W_{n-1}$, so that $W_n = W_0 \prod_{j=1}^n (-a_j) = (-1)^{n+1}\prod_{j=1}^n a_j$ with $W_0 = A_0 B_{-1} - A_{-1} B_0 = -1$. Since $a_j = -(2j-1)j^3 \neq 0$ for all $j \ge 1$, we have $-a_j > 0$ and hence $W_n \ne 0$ for all $n$, confirming that $\{A_n\}$ and $\{B_n\}$ are indeed linearly independent.
	
	\subsection{Proof of Convergence for the Ramanujan GCF}
	
	We now apply this criterion to the specific polynomial coefficients $a_n$ and $b_n$ defining the Ramanujan identity.
	
	\begin{proposition}[Structure of the Solution Basis and the Minimal Solution]\label{prop:minimal}
		The numerator sequence $\{A_n\}$ and the denominator sequence $\{B_n\}$ defined in Section~2 constitute a basis of the solution space for the recurrence $\mathcal{L}\,y = 0$.  Both are dominant solutions satisfying $A_n \sim (8/\pi^2)\,B_n$ as $n\to\infty$.  The true minimal solution is the linear combination
		\begin{equation}\label{eq:minimal_def}
			f_n \;=\; A_n - \frac{8}{\pi^2}\,B_n,
		\end{equation}
		which satisfies $\lim_{n\to\infty} f_n/B_n = 0$, confirming that the continued fraction converges.
	\end{proposition}
	
	\begin{proof}
		Consider the structure of the solutions derived via operator decoupling. From the boundary analysis in Phase 2, the numerator sequence is characterized by the vanishing of its auxiliary component ($w_n \equiv 0$), which forces $\{A_n\}$ to propagate strictly as the pure product trajectory:
		\begin{equation}
			A_n = \prod_{j=1}^{n+1} d_j.
		\end{equation}
		The denominator sequence, conversely, incorporates the propagation of the non-vanishing auxiliary term $w_0 = 1$, leading to the representation:
		\begin{equation}
			B_n = \left( \prod_{j=1}^{n+1} d_j \right) \left( \sum_{k=0}^{n} t_k \right).
		\end{equation}
		To determine the convergence behaviour, we examine the limit of the ratio between these two linearly independent solutions:
		\begin{equation}
			\textcolor{blue}{L \;:=\;} \lim_{n \to \infty} \frac{A_n}{B_n} = \lim_{n \to \infty} \frac{1}{\sum_{k=0}^{n} t_k}.
		\end{equation}
		As established in Phase 3, the general term $t_k$ satisfies the ratio test $\lim_{k \to \infty} |t_{k+1}/t_k| = 1/2 < 1$, which implies the absolute convergence of the series $S = \sum_{k=0}^{\infty} t_k$.
		
		Since $S = \pi^2/8$ is a finite positive constant, we obtain
		\begin{equation}
			L = \frac{1}{S} = \frac{8}{\pi^2},
		\end{equation}
		so $A_n \sim (8/\pi^2)\,B_n$ as $n\to\infty$.  In particular $\lim_{n\to\infty} A_n/B_n = 8/\pi^2 \neq 0$, which means that neither $\{A_n\}$ nor $\{B_n\}$ is the minimal solution; both are dominant solutions of equal asymptotic order.
		The true minimal solution is the linear combination $f_n = A_n - L\,B_n$ introduced in \eqref{eq:minimal_def}.  By linearity of the recurrence, $\{f_n\}$ satisfies $\mathcal{L}\,y=0$.  Moreover,
		\begin{equation}
			\frac{f_n}{B_n} = \frac{A_n}{B_n} - \frac{8}{\pi^2} \;\longrightarrow\; 0 \quad (n \to \infty),
		\end{equation}
		
		so $\{f_n\}$ is subdominant relative to $\{B_n\}$ and hence is the minimal solution in the sense of Pincherle's Theorem.  By that theorem, the existence of $\{f_n\}$ guarantees that the GCF converges, and its value is $\lim_{n\to\infty} A_n/B_n$, which we have already computed:
		\begin{equation}
			\mathcal{K} = S^{-1} = \left( \sum_{m=1}^{\infty} \frac{2^m}{m^2 \binom{2m}{m}} \right)^{-1} = \frac{8}{\pi^2}.
		\end{equation}
		The proof is complete.
	\end{proof}
	
	\begin{remark}[Worpitzky bound~{\cite[Thm.\,4.29]{wall1948analytic}; \cite[Thm.\,3.2.5]{lorentzen1992continued}}]
		The convergence of the GCF can also be confirmed independently via Worpitzky's theorem, which requires $\sup_{n \ge 1} |a_n / (b_n b_{n-1})| \le 1/4$. For our coefficients,
		\[
		\frac{|a_n|}{b_n b_{n-1}} = \frac{2n^4 - n^3}{(3n^2+3n+1)(3n^2-3n+1)},
		\]
		whose numerator has leading term $2n^4$ and whose denominator has leading term $9n^4$. Hence
		\[
		\lim_{n \to \infty} \frac{|a_n|}{b_n b_{n-1}} = \frac{2}{9} < \frac{1}{4},
		\]
		and a direct computation confirms that $|a_n|/(b_n b_{n-1}) \le 2/9 < 1/4$ for every $n \ge 1$. (The criterion $|a_n/(b_n b_{n-1})|\le 1/4$ is the form obtained from the original Worpitzky theorem for $\mathbf{K}(c_n/1)$ by applying the equivalence transformation $r_n = 1/b_n$, which maps the general fraction to one with unit partial denominators while preserving the value \cite[Thm.\,2.1]{wall1948analytic}.) Thus Worpitzky's theorem provides an independent elementary confirmation of absolute convergence.
	\end{remark}
	
	\section{Discussion and Conclusion}
	In this work, we have rigorously established the $8/\pi^2$ identity by demonstrating that the quartic polynomial continued fraction conjecture can be reduced to a tractable first-order cascade via shift operator decoupling. The convergence of the GCF is confirmed through Pincherle's Theorem: the true minimal solution of the associated difference equation is $f_n = A_n - (8/\pi^2)\,B_n$, which satisfies $f_n/B_n \to 0$ and thereby certifies convergence, while the convergent value itself is given directly by $\lim_{n\to\infty} A_n/B_n = 8/\pi^2$ as computed in Section~3. Our analysis reveals that the ratio $\rho = 1/2$ of the series $\sum t_k$ ensures absolute convergence of that series, and moreover, since $|a_n/(b_n b_{n-1})| \le 2/9 < 1/4$ for all $n \ge 1$ with $\lim_{n\to\infty} |a_n/(b_n b_{n-1})| = 2/9$, Worpitzky's classical criterion is also satisfied, providing an independent confirmation of the GCF's convergence.
	
	The methodology presented here suggests a broader applicability for the verification of automated conjectures involving higher-order polynomial coefficients. The discrete Green's function approach, enabled by the non-commutative operator decoupling, provides a robust framework for isolating transcendental constants within the kernel of variable-coefficient difference equations. These results suggest that the ``hidden'' algebraic structures discovered by the Ramanujan Machine are deeply rooted in the decoupling of operator rings, providing a systematic pathway for verifying automated conjectures of higher-order transcendental constants.

	\section*{Author contributions}
	The author is solely responsible for the conception, execution, and writing of this
	study.
	\section*{Data availability}
	No data was used for the research described in the article.
	\section*{Declarations}
	Competing interests The authors declare no competing interests.

	\section*{Acknowledgments}
	The authors received no financial support for the research, authorship, and/or publication of this article.
	\bibliography{sn-bibliography}

@article{Raayoni2021,
  author    = {Raayoni, Gal and Gottlieb, Shahar and Manor, Yoav and
               Pisha, George and Harris, Yahel and Mendlovic, Uri and
               Haviv, Doron and Hadad, Yaron and Kaminer, Ido},
  title     = {Generating conjectures on fundamental constants with the
               {Ramanujan} {Machine}},
  journal   = {Nature},
  volume    = {590},
  pages     = {67--73},
  year      = {2021},
  doi       = {10.1038/s41586-021-03229-4}
}

@book{wall1948analytic,
  author    = {Wall, Hubert Stanley},
  title     = {Analytic Theory of Continued Fractions},
  publisher = {D.\ Van Nostrand},
  address   = {New York},
  year      = {1948},
  note      = {Reprinted by Dover Publications, 2018.
               ISBN~978-0-486-83044-5}
}

@book{lorentzen1992continued,
  author    = {Lorentzen, Lisa and Waadeland, Haakon},
  title     = {Continued Fractions with Applications},
  series    = {Studies in Computational Mathematics},
  volume    = {3},
  publisher = {North-Holland},
  address   = {Amsterdam},
  year      = {1992},
  isbn      = {978-0-444-89265-2}
}

@article{Gautschi1967,
  author    = {Gautschi, Walter},
  title     = {Computational aspects of three-term recurrence relations},
  journal   = {SIAM Review},
  volume    = {9},
  number    = {1},
  pages     = {24--82},
  year      = {1967},
  doi       = {10.1137/1009002}
}

@article{Pincherle1894,
  author    = {Pincherle, Salvatore},
  title     = {Delle funzioni ipergeometriche e di varie questioni ad esse
               attinenti},
  journal   = {Giornale di Matematiche di Battaglini},
  volume    = {32},
  pages     = {209--291},
  year      = {1894}
}

@article{Lehmer1985,
  author    = {Lehmer, Derrick Henry},
  title     = {Interesting series involving the central binomial coefficient},
  journal   = {American Mathematical Monthly},
  volume    = {92},
  number    = {7},
  pages     = {449--457},
  year      = {1985},
  doi       = {10.2307/2322496}
}
	
\end{document}